\documentclass{tac}
\usepackage{mathrsfs}
\usepackage{tikz-cd}
\usepackage{amssymb}
\usepackage[bbgreekl]{mathbbol}
\usepackage{stmaryrd}
\usepackage{enumitem} 
\usepackage{xcolor}
\usepackage[colorlinks=true]{hyperref}
\hypersetup{allcolors=[rgb]{0.1,0.1,0.4}}

\DeclareMathAlphabet{\mathbbe}{U}{bbold}{m}{n}

\setlist{}
\setenumerate{leftmargin=*,labelindent=0\parindent}
\setitemize{leftmargin=\parindent}

\newtheorem{thm}{Theorem}[section]
\newtheorem{lem}[thm]{Lemma}
\newtheorem{prop}[thm]{Proposition}
\newtheorem{cor}[thm]{Corollary}

\theoremstyle{definition}
\newtheorem{defn}[thm]{Definition}
\newtheorem{ex}[thm]{Example}

\theoremstyle{remark}
\newtheorem{rmk}[thm]{Remark}
\newtheorem{exc}[thm]{Exercise}
\newtheorem{dig}[thm]{Digression}


\usepackage{chngcntr}
\counterwithin{equation}{section}

\newcommand{\op}{\mathrm{op}}

\newcommand{\id}{\mathrm{id}}
\newcommand{\ob}{\mathrm{ob}}

\mathrmdef{dom}
\mathrmdef{cod}
\mathrmdef{colim}
\mathrmdef{hocolim}
\mathrmdef{wcolim}
\mathrmdef{holim}

\newcommand{\cat}[1]{\textup{\textsf{#1}}}
\newcommand{\fun}[1]{\textup{#1}}

\newcommand{\KK}{\mathbb{K}}

\newcommand{\NN}{\mathbb{N}}

\newcommand{\eS}{\mathbb{S}}

\newcommand{\cA}{\mathsf{A}}
\newcommand{\cB}{\mathsf{B}}
\newcommand{\cC}{\mathsf{C}}
\newcommand{\cD}{\mathsf{D}}

\newcommand{\cS}{\mathsf{S}}

\newcommand{\DDelta}{\mathbbe{\Delta}}

\newcommand{\gC}{\mathfrak{C}}
\newcommand{\gN}{\mathfrak{N}}

\newcommand{\Space}{\mathbb{S}\cat{pace}}
\newcommand{\qCat}{\mathbb{Q}\cat{cat}}
\newcommand{\Mnd}{\mathbb{M}\cat{nd}}
\newcommand{\Adj}{\mathbb{A}\cat{dj}}
\newcommand{\Cat}{\mathbb{C}\cat{at}}
\newcommand{\hSpace}{\cat{h}\Space}
\newcommand{\sSet}{\cat{s}\mathbb{S}\cat{et}}

\newcommand{\Map}{\mathrm{Map}}
\newcommand{\Fun}{\mathrm{Fun}}

\newcommand{\Ho}{\cat{Ho}}
\newcommand{\sk}{\mathrm{sk}}

\title{Homotopy coherent structures}
\author{Emily Riehl}
\date{July 12-14, 2017; revised \today} 

\thanks{These lecture notes were prepared to accompany a three-hour mini course entitled ``Homotopy coherent structures'' delivered at the summer school accompanying the ``Floer homology and homotopy theory'' conference at UCLA supported by NSF Grant DMS-1563615. I am grateful to the organizers for this opportunity to speak,  to the members of the Australian Category Seminar who suffered through a preliminary version of these talks, and also for additional support from the National Science Foundation provided by the grant DMS-1551129 while these notes were being written. An anonymous referee made several suggestions that improved the exposition.
}

\address{Department of Mathematics\\Johns Hopkins University \\ 3400 N Charles Street \\ Baltimore, MD 21218}

\copyrightyear{2023}


\eaddress{eriehl@jhu.edu}

\begin{document}

\maketitle

\begin{abstract}
Naturally occurring diagrams in algebraic topology are commutative up to homotopy, but not on the nose. It was quickly realized that very little can be done with this information. Homotopy coherent category theory arose out of a desire to catalog the higher homotopical information required to restore constructibility (or more precisely, functoriality) in such ``up to homotopy'' settings. These notes provide a three-part introduction to homotopy coherent category theory. The first part surveys the classical theory of homotopy coherent diagrams of topological spaces. The second part introduces the homotopy coherent nerve and connects it to the free resolutions used to define homotopy coherent diagrams. This connection explains why diagrams valued in homotopy coherent nerves or more general $\infty$-categories are automatically homotopy coherent. The final part ventures into homotopy coherent algebra, connecting the newly discovered notion of homotopy coherent adjunction to the classical cobar and bar resolutions for homotopy coherent algebras.
\end{abstract}

\setcounter{tocdepth}{2}
\tableofcontents

\thirdleveltheorems

In algebraic topology and related settings it is common to encounter diagrams that do not commute on the nose but rather commute up to homotopy. Part \ref{part:I} of these notes opens with a motivating example. To perform further constructions with a homotopy commutative diagram, it is often necessary to make use of explicit homotopies that witness the homotopy commutativity together with higher-dimensional homotopies that encode various coherences between these witnesses. Together this data extends a \emph{homotopy commutative diagram} to a \emph{homotopy coherent diagram}, which is the primary example of a homotopy coherent structure.

Homotopy coherent category theory leads inexorably to the concept of \emph{quasi-categories}, first discovered by Boardman and Vogt when exploring the way in which homotopy coherent natural transformations compose. Traditionally,  homotopy coherent diagrams are valued in categories enriched over Kan complexes and indexed by simplicial categories constructed as free resolutions of ordinary 1-categories. In  Part \ref{part:II}, we observe that free resolutions are isomorphic to so-called \emph{homotopy coherent resolutions}, these being defined by the left adjoint to the more familiar homotopy coherent nerve functor. The homotopy coherent nerve takes a Kan complex enriched category to a quasi-category, and indeed, up to equivalence, all quasi-categories arise in this way. Thus, traditional homotopy coherent diagrams transpose to define maps of simplicial sets from the nerve of the indexing category into the quasi-category defined by the homotopy coherent. This observation connects the classical theory to a more modern setting for homotopy coherent category theory and explains the slogan that all diagrams valued in quasi-categories are homotopy coherent.

In Part \ref{part:III} we survey a small portion of homotopy coherent algebra by introducing homotopy coherent analogues of adjunctions, monads, and their algebras, notions which, somewhat surprisingly, can all be described using the simplicial categories formalism. We explain how these notions can be understood as part of ``model-independent'' $(\infty,1)$-category theory, applicable to more than just the quasi-categorical setting. This section concludes with suggestions for further reading that develop other topics in ``higher algebra'' that are not addressed here.

\section{Homotopy coherent diagrams}\label{part:I}

\subsection{Historical motivation}

If $X$ is a $G$-space, for $G$ a discrete group, and $Y$ is homotopy equivalent to $X$, then is $Y$ a $G$-space? The action of a group element $g \in G$ on $Y$ can be defined by transporting along the maps  $f \colon X \to Y$ and  $f^{-1} \colon Y \to X$ of the homotopy equivalence:
\[
\begin{tikzcd}
Y \arrow[r, "f^{-1}"] \arrow[d, dashed, "g_*"'] & X\arrow[d, "g_*"] \\ Y & X \arrow[l, "f"] 
\end{tikzcd}
\]
This defines a continuous endomorphism of $Y$ for every $g \in G$, as required by a $G$-action, but these maps are not necessarily automorphisms (since $f$ and $f^{-1}$ need not be homeomorphisms) nor is the composite of the actions associated to a pair of elements $g, h \in G$ equal to the action by their product: instead
\[ g_* \circ h_* = (f \circ g_* \circ f^{ -1}) \circ (f \circ h_* \circ f^{-1}) = f \circ g_* \circ (f^{-1} \circ f) \circ h_* \circ f^{-1} \] \[  \mathrm{and} \qquad (gh)_* = f \circ (gh)_* \circ f^{-1}\]
are homotopic via the homotopy $f^{-1} \circ f \simeq \id_X$.

So if $Y$ is not a $G$-space, then what is it? The main aim of Part \ref{part:I} is to develop language to describe this sort of situation. A $G$-space $X$ may be productively considered as a \emph{diagram} in the category of topological spaces indexed by a category $\cat{B}G$ with a single object and with an endomorphism corresponding to each element in the group.\footnote{Here ``diagram'' is synonymous with ``functor'': to define a functor whose domain is the category $\cat{B}G$ is to specify an image for the single object together with and endomorphism $g_*$ for each $g \in G$ so that $(gh)_*=g_*h_*$ and $e_* = \id$.}

By contrast the ``up to homotopy $G$-space'' $Y$ is instead a \emph{homotopy commutative diagram}. Modulo point-set topological considerations that we sweep under the rug using a technique that will be described in \S\ref{sec:shape}, the category of topological spaces is self-enriched, meaning that the set of continuous functions between any pair of spaces $X$ and $Y$ is itself a space, which we denote by $\Map(X,Y)$. The points in $\Map(X,Y)$ are the continuous functions $f \colon X \to Y$ while a path in $\Map(X,Y)$ between points $f$ and $g$ is a \textbf{homotopy}. Two parallel maps $f,g \colon X \to Y$ are \textbf{homotopic} --- in symbols ``$f \simeq g$'' --- just when they are in the same path component in $\pi_0\Map(X,Y)$. The space $\Map(X,Y)$ also has higher structure, the paths between paths, and the paths between paths between paths and so on, which define higher homotopies.

 Importantly, composition of continuous functions itself defines a continuous function between mapping spaces
\[ 
\begin{tikzcd}
\Map(Y,Z) \times \Map(X,Y) \arrow[r, "\circ"] & \Map(X,Z)
\end{tikzcd}
\] so this relation of taking homotopy classes of maps is preserved by pre- and post-composition with another continuous function. This defines the category $\hSpace$ of spaces and homotopy classes of continuous functions as a quotient of the enriched category $\Space$ of spaces and mapping spaces $\Map(X,Y)$.


\begin{defn}
If $\cA$ is an ordinary category then
\begin{itemize}
\item a \textbf{diagram of spaces} is just a functor $\cA \to \Space$
\item a \textbf{homotopy commutative diagram of spaces} is a functor $\cA \to \hSpace$
\end{itemize}
\end{defn}

Thus, the $G$-space $X$ defines a diagram $X \colon \cat{B}G \to \Space$ while the ``up-to-homotopy'' $G$-space $Y$ defines a homotopy commutative diagram $Y \colon \cat{B}G \to \hSpace$. This terminology suggests a related question: can every ``up-to-homotopy'' $G$-space be rectified to a homotopy equivalent $G$-space? Or more generally, when does a homotopy commutative diagram $F \colon \cA \to \hSpace$ admit a \textbf{rectification}, i.e., a strictly-commutative diagram $F' \colon \cA \to \Space$ together with homotopy equivalences $Fa \simeq F'a$  that define a natural transformation\footnote{See Definition \ref{defn:nat-trans}.} in $\hSpace$?

\begin{thm}[{\cite[2.5]{DKS-homotopy}}] A homotopy commutative diagram has a rectification if and only if it may be lifted to a homotopy coherent diagram. Moreover, equivalence classes of rectifications correspond bijectively with equivalences classes of homotopy coherent diagrams.
\end{thm}
\begin{proof}
This result is proven as a corollary of \cite[2.4]{DKS-homotopy} which demonstrates that appropriately defined spaces of homotopy coherent diagrams and rectifications (defined slightly differently than above) are weak homotopy equivalent.
\end{proof}

In particular, since the homotopy commutative $G$-space $Y$ is homotopy equivalent to the strict $G$-space $X$, it must underlie a homotopy coherent diagram of shape $\cat{B}G$. Our task is now to work out what exactly the phrase \emph{homotopy coherent diagram} means.

\subsection{The shape of a homotopy coherent diagram}\label{sec:shape}

To build intuition for the general notion of a homotopy coherent diagram, it is helpful to consider a special case. To that end, let \[
\begin{tikzcd}
\bbomega := & 0 \arrow[r] & 1 \arrow[r] & 2 \arrow[r] & 3 \arrow[r] & \cdots
\end{tikzcd}
\]
denote the category whose objects are finite ordinals and with a unique morphism $j \to k$ if and only if $j \leq k$. 

An $\bbomega$-shaped graph in $\Space$ is comprised of spaces $X_k$ for each $k \in \bbomega$ together with morphisms $f_{j,k} \colon X_j \to X_k$ whenever $j < k$.\footnote{To simplify somewhat we adopt the convention that $f_{j,j}$ is the identity, making this data into a \textbf{reflexive directed graph} with implicitly designated identities.} This data defines a \emph{homotopy commutative diagram} $\bbomega \to \hSpace$  just when $f_{i,k} \simeq f_{j,k} \circ f_{i,j}$ whenever $i < j < k$.

To extend this data to a \emph{homotopy coherent diagram} $\bbomega \to \Space$ requires:
\begin{itemize}
\item Chosen homotopies $h_{i,j,k} \colon f_{i,k} \simeq f_{j,k} \circ f_{i,j}$ whenever $i < j < k$. This amounts to specifying a path in $\Map(X_i,X_k)$ from the vertex $f_{i,k}$ to the vertex $f_{j,k} \circ f_{i,j}$, which is obtained as the composite of the two vertices $f_{i,j} \in \Map(X_i,X_j)$ and $f_{j,k} \in \Map(X_j,X_k)$.
\item For $i < j < k < \ell$, the chosen homotopies provide four paths in $\Map(X_i,X_\ell)$
\[
\begin{tikzcd}[column sep=large]
f_{i,\ell} \arrow[r, no head, "h_{i,k,\ell}"] \arrow[d, no head, "h_{i,j,\ell}"'] & f_{k,\ell} \circ f_{i,k} \arrow[d, no head, "f_{k, \ell} \circ h_{i,j,k}"] \\ f_{j,\ell} \circ f_{i,j} \arrow[r, no head, "h_{j,k, \ell} \circ f_{i,j}"'] & f_{k,\ell} \circ f_{j,k} \circ f_{i,j}
\end{tikzcd} 
\] We then specify a higher homotopy --- a \emph{2-homotopy} --- filling in this square.
\item For $i < j < k < \ell < m$, the previous choices provide 12 paths and six 2-homotopies in $\Map(X_i,X_m)$ that assemble into the boundary of a cube. We then specify a \emph{3-homotopy}, a homotopy between homotopies between homotopies, filling in this cube.
\item Etc.
\end{itemize}

Even in this simple case of the category $\bbomega$, this data is a bit unwieldy. Our task is to define a category to index this homotopy coherent data arising from $\bbomega$: the objects $X_i$, the functions $X_i \to X_j$, the 1-homotopies $h_{i,j,k}$, the 2-homotopies, and so on. This data will assemble into a \emph{simplicial category} whose objects are the same as the objects of $\bbomega$ but which will have $n$-morphisms in each dimension $n \geq 0$, to index the $n$-homotopies. 

Because of the convenience of the mechanism of simplicial categories, and to avoid the point-set topology considerations alluded to above, we should now come clean and admit that we prefer to assume that our ``mapping spaces'' are really simplicial sets, or more precisely \emph{Kan complexes}. We defer the latter definition to the next section but give the others now.

\begin{dig}[a crash course on simplicial sets]\label{dig:sset}
There is a convenient category whose objects model topological spaces (at least up to weak homotopy type): the category $\cat{sSet}$ of simplicial sets. A \textbf{simplicial set} $X$ is a graded set $(X_n)_{n  \geq 0}$---with the elements of $X_n$ called \textbf{n-simplices}---together with maps 
\begin{equation}\label{eq:simp-set}
\begin{tikzcd}
X := &
X_0 \arrow[r, tail] & \arrow[l, shift left=0.75em, two heads] \arrow[l, shift right=0.75em, two heads]  X_1  \arrow[r, shift left=0.75em, tail] \arrow[r, shift right=0.75em, tail] &   \arrow[l, two heads] \arrow[l, two heads, shift right=1.5em] \arrow[l, two heads, shift left=1.5em]   X_2 \arrow[r, tail] \arrow[r, tail, shift right=1.5em] \arrow[r, tail, shift left=1.5em]   &\arrow[l, two heads, shift right=0.75em] \arrow[l, two heads, shift right=2.25em] \arrow[l, two heads, shift left=2.25em]  \arrow[l, two heads, shift left=0.75em] X_3  & \cdots
\end{tikzcd}
\end{equation}
that fulfill two functions: 
\begin{itemize}
\item The $n+1$ \textbf{face maps} $X_n \twoheadrightarrow X_{n-1}$ identify the faces of an $n$-simplex.
\item The $n$ \textbf{degeneracy maps} $X_n \rightarrowtail X_{n+1}$ define degenerate $n+1$-simplices that project onto a given $n$-simplex.
\end{itemize}

There is a slick way to make all of this precise. Let $\DDelta$ denote the category of finite non-empty ordinals $[n]= \{0,1\ldots, n\}$ and order-preserving maps. The maps in the category $\DDelta$ are generated under composition by the basic inclusions and surjections displayed here:
\[
\begin{tikzcd}
\DDelta := &
{[0]} \arrow[r, shift left=0.75em, tail] \arrow[r, shift right=0.75em, tail] & {[1]} \arrow[l, two heads] \arrow[r, tail] \arrow[r, tail, shift right=1.5em] \arrow[r, tail, shift left=1.5em]  & {[2]} \arrow[l, shift left=0.75em, two heads] \arrow[l, shift right=0.75em, two heads]  \arrow[r, tail, shift right=0.75em] \arrow[r, tail, shift right=2.25em] \arrow[r, tail, shift left=2.25em]  \arrow[r, tail, shift left=0.75em] & {[3]} \arrow[l, two heads] \arrow[l, two heads, shift right=1.5em] \arrow[l, two heads, shift left=1.5em]   & \cdots
\end{tikzcd}
\]
Now a \textbf{simplicial set} is just a  functor $X \colon \DDelta^\op \to \cat{Set}$.

The category $\cat{sSet}$ is generated by the standard $n$-simplices $\Delta^n$ for each $n \geq 0$. The standard $n$-simplex may be thought of as an (ordered) $n$-simplex spanned by the vertices $0,\ldots, n$. More precisely, $\Delta^n$ is the functor represented by the object $[n] \in \DDelta$; that is an $m$-simplex in $\Delta^n$ is a map $\alpha \colon [m] \to [n]$ in $\DDelta$. There are various maps between these standard simplices 
\[
\begin{tikzcd}
\DDelta := &
\Delta^0 \arrow[r, shift left=0.75em, tail] \arrow[r, shift right=0.75em, tail] & \Delta^1 \arrow[l, two heads] \arrow[r, tail] \arrow[r, tail, shift right=1.5em] \arrow[r, tail, shift left=1.5em]  & \Delta^2 \arrow[l, shift left=0.75em, two heads] \arrow[l, shift right=0.75em, two heads]  \arrow[r, tail, shift right=0.75em] \arrow[r, tail, shift right=2.25em] \arrow[r, tail, shift left=2.25em]  \arrow[r, tail, shift left=0.75em] & \Delta^3 \arrow[l, two heads] \arrow[l, two heads, shift right=1.5em] \arrow[l, two heads, shift left=1.5em]   & \cdots
\end{tikzcd}
\]
each of the maps denoted by ``$\rightarrowtail$'' given by an ordered injection of the vertices and each of the maps denoted ``$\twoheadrightarrow$'' given by an ordered surjection of the vertices. By the Yoneda lemma, $\DDelta$ is isomorphic to the full subcategory of $\cat{sSet}$ spanned by the standard simplices; the functor $\Delta^\bullet \colon \DDelta \hookrightarrow \cat{sSet}$ sending $[n]$ to the standard $n$-simplex $\Delta^n$ is referred to as the \emph{Yoneda embedding}.

The standard simplices and maps between them generate the category $\cat{sSet}$ under gluing. That is, any simplicial set $X$ may be thought of as a triangulated space built by gluing in a standard simplex $\Delta^n$ for each $n$-simplex of $X$ along the face and degeneracy maps (\ref{eq:simp-set}). These intuitions can be formalized via a comparison adjunction 
    \begin{equation}\label{eq:singular-adjunction}
\begin{tikzcd}
\cat{sSet} \arrow[r,bend left, "{|-|}"] \arrow[r, phantom, "\perp"]  & \cat{Top} \arrow[l, bend left, "\fun{Sing}"]
\end{tikzcd}
\end{equation}
obtained by applying Kan's construction \cite[1.5.1]{riehl-cathtpy} to the functor $\Delta^\bullet \colon \DDelta \to \cat{Top}$ that sends $[n] \in \DDelta$ to the geometric $n$-simplex $|\Delta^n|$. The right adjoint carries a space $Y$ to its \textbf{total singular complex}, the simplicial set whose $n$-simplices are continuous maps $|\Delta^n| \to Y$. The left adjoint carries a simplicial set $X$ to its \textbf{geometric realization}, a simplicial complex formed by gluing in a topological $n$-simplex for each $n$-simplex of $X$ along the face and degeneracy maps.\footnote{This adjunction defines an ``equivalence of homotopy theories'' in a sense made precise by the notion of Quillen equivalence between model categories \cite[\S II.3]{quillen}.}
For a more leisurely introduction to simplicial sets, see \cite{riehl-leisurely}.
\end{dig}

A simplicial category is typically thought of as a category with objects together with mapping spaces (i.e., simplicial sets) between them. There is an alternate presentation of this data which will also be convenient in which an $n$-\emph{simplex} in a mapping space from $a$ to $b$ is encoded as an $n$-\emph{arrow} from $a$ to $b$.

\begin{defn} A \textbf{simplicial category} $\cA_\bullet$ is given by categories $\cA_n$ for each $n \geq 0$ with a common set of objects $\ob\cA$ and whose morphisms are called $n$-\textbf{arrows} that assemble into a diagram $\DDelta^\op \to \cat{Cat}$ of identity-on-objects functors.
\begin{equation}\label{eq:simp-category-shape}
\begin{tikzcd}
\cA_\bullet := &
\cA_0 \arrow[r, tail] & \arrow[l, shift left=0.75em, two heads] \arrow[l, shift right=0.75em, two heads]  \cA_1  \arrow[r, shift left=0.75em, tail] \arrow[r, shift right=0.75em, tail] &   \arrow[l, two heads] \arrow[l, two heads, shift right=1.5em] \arrow[l, two heads, shift left=1.5em]   \cA_2 \arrow[r, tail] \arrow[r, tail, shift right=1.5em] \arrow[r, tail, shift left=1.5em]   &\arrow[l, two heads, shift right=0.75em] \arrow[l, two heads, shift right=2.25em] \arrow[l, two heads, shift left=2.25em]  \arrow[l, two heads, shift left=0.75em] \cA_3  & \cdots
\end{tikzcd}
\end{equation}
\end{defn}

\begin{prop} The following are equivalent:
\begin{itemize}
\item a simplicial category $\cA_\bullet$ with object set $\ob\cA$
\item a simplicially enriched category $\cA$ with objects $\ob\cA$
\end{itemize}
\end{prop}
\begin{proof}
For any $x,y \in \ob\cA$, an $n$-arrow in $\cA_n(x,y)$ corresponds to an $n$-simplex in the mapping space $\cA(x,y)$.
\end{proof}

For any ordinary category $\cA$, we now introduce a simplicial category $\gC\cA$ whose $n$-arrows parametrize the data of a homotopy coherent diagram of shape $\cA$.\footnote{Cordier and Porter \cite{CP-vogt} write $\mathbf{S}(\cA)$ for the simplicial category $\gC\cA$. Here we use notation that some readers might recognize from a related context. In Theorem \ref{thm:resolution-equals-realization}, we prove that these two objects are isomorphic.}

\begin{defn}[free resolutions]
Forgetting composition, let $U\cA$ denote the underlying reflexive directed graph of a category $\cA$, and let $FU\cA$ denote the free category on the underlying reflexive directed graph of $\cA$. It has the same objects as $\cA$ and its non-identity arrows are strings of composable non-identity arrows of $\cA$. 

We define a simplicial category $\gC\cA_\bullet$ with $\ob\gC\cA=\ob\cA$ and with the category of $n$-arrows $\gC\cA_n := (FU)^{n+1}\cA$. A non-identity $n$-arrow is a string of composable arrows in $\cA$ with each arrow in the string enclosed in exactly $n$ pairs of well-formed parentheses. In the case $n=0$, this recovers the previous description of the non-identity 0-arrows in $FU\cA$, strings of composable non-identity arrows of $\cA$. 

It remains to define the required identity-on-objects functors:\footnote{More concisely, the free and forgetful functors just described define an adjunction
\[
\begin{tikzcd}[ampersand replacement=\&]
\cat{rDirGph} \arrow[r,bend left, "F"] \arrow[r, phantom, "\perp"]  \& \cat{Cat} \arrow[l, bend left, "U" pos=0.5]
\end{tikzcd}
\]
between small categories and reflexive directed graphs inducing a comonad $FU$ on $\cat{Cat}$; see Definition \ref{defn:adjunction}. The simplicial object $\gC\cA_\bullet$ is defined by evaluating the comonad resolution for $(FU, \epsilon, F\eta U)$ on a small category $\cA$. The face and degeneracy maps are whiskerings of the unit and counit of the adjunction; hence the notation. This structure will reappear in Part \ref{part:III} below.}
\[
\begin{tikzcd}
\gC\cA_\bullet := &
FU\cA\arrow[r, tail] & \arrow[l, shift left=0.75em, two heads] \arrow[l, shift right=0.75em, two heads]  (FU)^2\cA  \arrow[r, shift left=0.75em, tail] \arrow[r, shift right=0.75em, tail] &   \arrow[l, two heads] \arrow[l, two heads, shift right=1.5em] \arrow[l, two heads, shift left=1.5em]   (FU)^3\cA \arrow[r, tail] \arrow[r, tail, shift right=1.5em] \arrow[r, tail, shift left=1.5em]   &\arrow[l, two heads, shift right=0.75em] \arrow[l, two heads, shift right=2.25em] \arrow[l, two heads, shift left=2.25em]  \arrow[l, two heads, shift left=0.75em] (FU)^4\cA  & \cdots
\end{tikzcd}
\]
For $j \geq 1$, the face maps \[(FU)^k\epsilon(FU)^j  \colon (FU)^{k+j+1}\cA \to (FU)^{k+j}\cA\] remove the parentheses that are contained in exactly $k$ others, while $FU\cdots FU\epsilon$ composes the morphisms inside the innermost parentheses. For $j \geq 1$, the degeneracy maps \[F(UF)^k\eta(UF)^jU \colon (FU)^{k+j+1}\cA \to (FU)^{k+j+2}\cA\] double up the parentheses that are contained in exactly $k$ others, while $F\cdots UF\eta U$ inserts parentheses around each individual morphism.
\end{defn}

\begin{ex} In the case of a discrete group $G$ regarded as a one-object category $\cat{B}G$, the free resolution $\gC{\cat{B}G}$ is defined by specifying the single endo-hom-set of each category $(FU)^n\cat{B}G$, together with the composition action. The underlying graph of $\cat{B}G$ is given by the non-identity elements of $G$, and thus $(FU)\cat{B}G$ is the group of words in these letters, i.e., the free group on the non-identity elements of $G$. The group $(FU)^2\cat{B}G$ is then the group of words of words and so on. 
\end{ex} 

\begin{exc}\label{exc:omega} Compute $\gC\bbomega$ and show that its $n$-arrows enumerate the data described at the introduction to this section.\footnote{In fact, it has more $n$-arrows than the $n$-homotopies describe above. We will be able to explain this when we return to this simplicial category in Example \ref{ex:coherent-simplex}.}
\end{exc}

The category $\cA$ can also be thought of as a discrete simplicial category in which the diagram (\ref{eq:simp-category-shape}) is constant at $\cA$, so the only $n$-arrows are degenerated 0-arrows. There is a canonical ``augmentation'' map $\epsilon \colon \gC\cA \to \cA$ determined by its degree zero component $\epsilon \colon FU\cA \to \cA$ which is just given by composition in $\cA$.

\begin{prop} The functor $\epsilon \colon\gC\cA \to \cA$ is a local homotopy equivalence of simplicial sets. That is, for any pair of objects $x,y \in \cA$, the map $\epsilon \colon \gC\cA(x,y) \to \cA(x,y)$ is a homotopy equivalence: $\gC\cA(x,y)$ is homotopy equivalent to the discrete set $\cA(x,y)$ of arrows in $\cA$ from $x$ to $y$.
\end{prop}
\begin{proof}
The augmented simplicial object
\[
\begin{tikzcd}
\cA \arrow[r, bend right, dashed] &
(FU)\cA\arrow[r, tail] \arrow[l, two heads] \arrow[r, bend right=35, dashed]& \arrow[l, shift left=0.75em, two heads] \arrow[l, shift right=0.75em, two heads]  (FU)^2\cA  \arrow[r, shift left=0.75em, tail] \arrow[r, shift right=0.75em, tail] \arrow[r, bend right=50, dashed] &   \arrow[l, two heads] \arrow[l, two heads, shift right=1.5em] \arrow[l, two heads, shift left=1.5em]   (FU)^3\cA \arrow[r, tail] \arrow[r, tail, shift right=1.5em] \arrow[r, tail, shift left=1.5em] \arrow[r, bend right=75, dashed]  &\arrow[l, two heads, shift right=0.75em] \arrow[l, two heads, shift right=2.25em] \arrow[l, two heads, shift left=2.25em]  \arrow[l, two heads, shift left=0.75em] (FU)^4\cA  & \cdots
\end{tikzcd}
\] 
is split at the level of reflexive directed graphs (i.e., after applying $U$). These splittings are not functors, but that won't matter. These directed graph morphisms displayed here are all identity on objects, which means that for any $x,y$ there is a split augmented simplicial set
\[
\begin{tikzcd}[column sep=small]
\cA(x,y) \arrow[r, bend right] &
(FU)\cA(x,y)\arrow[r, tail] \arrow[l, two heads] \arrow[r, bend right=35]& \arrow[l, shift left=0.75em, two heads] \arrow[l, shift right=0.75em, two heads]  (FU)^2\cA(x,y)  \arrow[r, shift left=0.75em, tail] \arrow[r, shift right=0.75em, tail] \arrow[r, bend right=50] &   \arrow[l, two heads] \arrow[l, two heads, shift right=1.5em] \arrow[l, two heads, shift left=1.5em]   (FU)^3\cA(x,y) \arrow[r, tail] \arrow[r, tail, shift right=1.5em] \arrow[r, tail, shift left=1.5em] \arrow[r, bend right=75]  &\arrow[l, two heads, shift right=0.75em] \arrow[l, two heads, shift right=2.25em] \arrow[l, two heads, shift left=2.25em]  \arrow[l, two heads, shift left=0.75em] (FU)^4\cA(x,y) \cdots
\end{tikzcd}
\] 
and now some classical simplicial homotopy theory proves the claim \cite{meyer}.
\end{proof}

\subsection{Homotopy coherent diagrams and homotopy coherent natural transformations}

Finally, we can give a precise definition of the key notion of a homotopy coherent diagram:

\begin{defn} A \textbf{homotopy coherent diagram} of shape $\cA$ is a functor $\gC\cA \to \Space$.
\end{defn}

\begin{ex}
A strictly commutative diagram $F\colon \cA\to\Space$ gives rise to a homotopy coherent diagram by composing with the augmentation map
\begin{equation}\label{eq:htpy-from-strict}
\begin{tikzcd}
\gC\cA \arrow[r, "\epsilon"] & \cA \arrow[r, "F"] & \Space.
\end{tikzcd}
\end{equation}
In this case, all $n$-homotopies are identities.
\end{ex}

Not every homotopy commutative diagram can be made homotopy coherent. The following counterexample was suggested by Thomas Kragh and communicated by Hiro Tanaka.

\begin{ex}
Let $p \colon E \to B$ be a Serre fibration with $i \colon F \to E$ the inclusion of the fiber over the basepoint $*$ of $B$. A diagram
\[ 
\begin{tikzcd} & X \arrow[dl, "f"'] \arrow[dr, "e"]  \arrow[dd, "\ast" near end] \\ F \arrow[rr, "i" near start, crossing over] \arrow[dr, "\ast"'] & & E \arrow[dl, "p"] \\ & B
\end{tikzcd}
\]
is \emph{homotopy commutative} if there exist homotopies $\alpha \colon e \simeq if$ and $\beta \colon pe \simeq *$, the other two triangles being strictly commutative. The diagram is then \emph{homotopy coherent} if and only if there exists a 2-homotopy between $p\alpha \colon pe \simeq *$ and $\beta$. If this is the case, then since $p$ is a Serre fibration it is possible to lift the $2$-homotopy along $p$ to define a homotopy $\alpha' \colon e \simeq if$ so that $p\alpha' = \beta$. Applying the universal property of the fiber $F$ as the homotopy pullback of $p$ along the inclusion of the basepoint, the homotopy $\beta$ induces a map $g \colon X \to F$ and the homotopy $\alpha'$ then implies that $f$ and $g$ are homotopic.

Applying these observations in the case of the Hopf fibration, consider the diagram 
\[ 
\begin{tikzcd} & S^1 \arrow[dl, "n"'] \arrow[dr, "i"]  \arrow[dd, "\ast" near end] \\ S^1 \arrow[rr, "i" near start, crossing over] \arrow[dr, "\ast"'] & & S^3 \arrow[dl, "p"] \\ & S^2
\end{tikzcd}
\]
involving a map $n \colon S^1 \to S^1$ of degree $n \neq 1$. Since $\pi_1 S^3=0$, there exists a homotopy $\alpha \colon i \simeq in$.  Both $pi$ and $pin$ equal the constant map $*$ at the basepoint of $S^2$, but $p\alpha$ is not 2-homotopic to the constant homotopy $*$, for if it were we would obtain a homotopy between the map of degree $n$ and the identity map $S^1 \to S^1$. Thus, this homotopy commutative diagram cannot be made homotopy coherent.
\end{ex}

A natural transformation is a type of higher morphism between parallel functors. Natural transformations are analogous to homotopies with the category $[1] = 0 \to 1$ playing the role of the interval. 

\begin{defn}\label{defn:nat-trans} Given a parallel pair of functors $F,G \colon \cC \to \cD$, a \textbf{natural transformation} $\alpha \colon F \to G$ is specified by a functor $\alpha \colon \cC \times [1] \to \cD$ that restricts on the ``endpoints'' of $[1]$ to $F$ and $G$ as follows:
\[
\begin{tikzcd} \cC \arrow[d, tail, "0"'] \arrow[dr, "F"] \\ \cC \times [1] \arrow[r, "\alpha"]  & \cD \\ \cC \arrow[u, tail, "1"] \arrow[ur, "G"']
\end{tikzcd}
\]
\end{defn}

This suggests the following definition of a homotopy coherent natural transformation.

\begin{defn} A \textbf{homotopy coherent natural transformation} $\alpha \colon F \to G$ between homotopy coherent diagrams $F$ and $G$ of shape $\cA$ is a homotopy coherent diagram of shape $\cA \times [1]$ that restricts on the endpoints of $[1]$ to $F$ and $G$ as follows:
\[
\begin{tikzcd} \gC\cA \arrow[d, tail, "0"'] \arrow[dr, "F"] \\ \gC(\cA \times [1]) \arrow[r, "\alpha"]  & \Space \\ \gC\cA \arrow[u, tail, "1"] \arrow[ur, "G"']
\end{tikzcd}
\]
\end{defn}

Note that the data of a pair of homotopy coherent natural transformations $\alpha \colon F \to G$ and $\beta \colon G \to H$ between homotopy coherent diagrams of shape $\cA$ does not uniquely determine a (vertical) ``composite'' homotopy coherent natural transformation $F \to H$ because this data does not define a homotopy coherent diagram of shape $\cA \times [2]$, where $[2] = 0 \to 1 \to 2$.\footnote{In notation to be introduced in Part \ref{part:II}, $\alpha$ and $\beta$ define a diagram of shape $\gC(\cA\times \Lambda^2_1)$ rather than a diagram of shape $\gC(\cA\times[2])$, where $\Lambda^2_1$ is the shape of the generating reflexive directed graph of the category $[2]$.} This observation motivated Boardman and Vogt to define, in place of a \emph{category} of homotopy coherent diagrams and natural transformations of shape $\cA$, a \emph{quasi-category} of homotopy coherent diagrams and natural transformations of shape $\cA$ \cite[\S IV.2]{BV}. 

\begin{defn} For any category $\cA$, let $\fun{Coh}(\cA,\Space)$ denote the simplicial set whose $n$-simplices are homotopy coherent diagrams of shape $\cA \times [n]$, i.e., are simplicial functors
\[ \gC(\cA\times[n]) \to \Space,\] where $[n] \subset \bbomega$ denotes the category freely generated by the reflexive directed graph
\[
\begin{tikzcd}
{[n]} := & 0 \arrow[r] & 1 \arrow[r] & 2 \arrow[r] & 3 \arrow[r] & \cdots \arrow[r] & n.
\end{tikzcd}
\]
\end{defn}

The simplicial category $\Space$ has an important property alluded to above: its mapping spaces $\Map(X,Y)$ are \textbf{Kan complexes}, simplicial sets in which any horn 
\[
\begin{tikzcd}
\Lambda^n_k \arrow[r] \arrow[d, hook] & \Map(X,Y) \\ \Delta^n \arrow[ur, dashed]
\end{tikzcd}
\]
with $0 \leq k \leq n$ may be filled to a simplex. Any simplicial category, such as $\Space$, extracted from a topologically-enriched category is automatically \emph{Kan complex enriched} because its mapping spaces are defined as total singular complexes of topological spaces \cite[\S 16.1]{riehl-cathtpy}. It is because of this property that:

\begin{thm}[{\cite[4.9]{BV}}] $\fun{Coh}(\cA,\Space)$ is a \textbf{quasi-category}, i.e., any inner horn
\[
\begin{tikzcd}
\Lambda^n_k \arrow[r] \arrow[d, hook] & \fun{Coh}(\cA,\Space) \\ \Delta^n \arrow[ur, dashed]
\end{tikzcd}
\]
with $0  < k <n$ admits a filler.
\end{thm}
\begin{proof} This can be checked directly, or deduced as an immediate consequence --- see Corollary \ref{cor:coherent-diagram-quasi-category} --- of a result that we will prove in Part \ref{part:II}.
\end{proof}

Boardman and Vogt referred to inner horn filling as the ``restricted Kan condition.'' Joyal introduced the name ``quasi-categories'' for these ``weak Kan complexes.'' Quasi-categories define a popular model of $(\infty,1)$-\emph{categories}, categories weakly enriched in topological spaces, about more which in Part \ref{part:III}.

Any quasi-category has a \textbf{homotopy category} whose objects are the vertices and whose morphisms are 1-simplices up to a homotopy relation $f \simeq g$ between parallel 1-simplices $f,g \colon x \to y$ witnessed by a 2-simplex with boundary:
\[
\begin{tikzcd} & y \arrow[dr, equals] \\ x \arrow[ur, "f"] \arrow[rr, "g"'] & & y
\end{tikzcd}
\]
Composition relations are also witnessed by 2-simplices: the homotopy class of $f \colon x \to y$ and the homotopy class of $g \colon y \to z$ compose to the homotopy class of $h \colon x \to z$ if and only if there is a 2-simplex whose boundary has the form
\[
\begin{tikzcd} & y \arrow[dr, "g"] \\ x \arrow[ur, "f"] \arrow[rr, "h"'] & & z
\end{tikzcd}
\]

The following result was first proven by Vogt and then generalized by Cordier and Porter:

\begin{thm}[{\cite{vogt1,CP-vogt}}]\label{thm:CPV}
The natural map $\Space^\cA \to \fun{Coh}(\cA,\Space)$ defined by (\ref{eq:htpy-from-strict}) induces an equivalence of homotopy categories
\[ \begin{tikzcd} \Ho(\Space^\cA) \arrow[r, "\simeq"] &  \Ho\fun{Coh}(\cA,\Space),\end{tikzcd}\] where $\Ho(\Space^\cA)$ is defined by localizing at the componentwise homotopy equivalences.
\end{thm}

An equivalence of categories is in particular essentially surjective. Theorem \ref{thm:CPV} tells us that any homotopy coherent diagram of shape $\cA$ is equivalent, via a homotopy coherent natural transformation, to a strictly commutative diagram of shape $\cA$. This result answers the rectification problem for homotopy commutative diagrams.

\section{Homotopy coherent realization and the homotopy coherent nerve}\label{part:II}

Recall that a homotopy coherent diagram of shape $\cA$ is a simplicial functor indexed by a category $\gC\cA$ defined as a free resolution of $\cA$, a construction that will be reviewed momentarily. Explicitly, the data of such a diagram is comprised of objects $X_a$ for each object $a \in \cA$ plus maps of simplicial sets
\[ \gC\cA(a,b) \to \Map(X_a,X_b)\] for each pair of objects that are functorial in the sense of commuting with the composition functions:
\[ 
\begin{tikzcd}
\gC\cA(b,c) \times \gC\cA(a,b) \arrow[r, "\circ"] \arrow[d] & \gC\cA(a,c) \arrow[d] \\ \Map(X_b,X_c) \times \Map(X_a,X_b) \arrow[r, "\circ"] & \Map(X_a,X_c)
\end{tikzcd}
\]
 Previously we interpreted $X_a$ and $X_b$ as spaces, but this interpretation is actually not necessary. What we do need is for $\Map(X_a,X_b)$ to be a ``space,'' by which we mean a Kan complex,  because it is in these mapping spaces that we are defining homotopies (as 1-simplices) and higher homotopies (as higher simplices). So henceforth, we will extend our notion of \textbf{homotopy coherent diagram} to encompass any simplicial functor  $\gC\cA \to \eS$ whose codomain $\eS$ is a category enriched in Kan complexes.\footnote{Categories enriched in Kan complexes are called ``locally Kan'' in much of the literature.} Any topological category can be made into a category enriched in Kan complexes by applying the total singular complex functor (\ref{eq:singular-adjunction}) to its mapping spaces. In particular, the category $\Space$ becomes a Kan complex enriched category in this way, but there are many other examples as well.

Indeed, many (large) quasi-categories --- e.g., of spaces, of spectra, or whatever ---  originate as categories enriched in Kan complexes by Theorem \ref{thm:nerve-to-quasi} below. Our aim today is to explain how diagrams that are valued in quasi-categories that arise in this way are automatically homotopy coherent.

To justify this slogan, we offer a second perspective on the simplicial category $\gC\cA$ defined as the free resolution of a category $\cA$, explaining its relationship to the famous \emph{homotopy coherent nerve} functor. This work will also allow us to generalize the indexing shapes for homotopy coherent diagrams to encompass simplicial sets  which may or may not be nerves of categories. This will allow us to distinguish, e.g., between the ordinal category $[2] = 0 \to 1 \to 2$ and its generating reflexive directed graph $\Lambda^2_1$.

\subsection{Free resolutions are simplicial computads}

An arrow in a category is \textbf{atomic}  if it is not an identity and if it admits no non-trivial factorizations, i.e., if whenever $f=g \circ h$  then one or other of $g$ and $h$ is an identity. A category is \textbf{freely
    generated} by a reflexive directed graph of atomic arrows if and only if each of its
  non-identity arrows may be uniquely expressed as a composite of atomic arrows.\footnote{This is the case just when the category is in the essential image of the free category functor $F \colon \cat{rDirGph} \to \cat{Cat}$.} 
  
  The following definition is due to Dwyer and Kan \cite[4.5]{DK-simplicial} who used the terminology ``free simplicial categories.''\footnote{The reader familiar with model categorical intuition might find it helpful to note that the simplicial computads are precisely the cofibrant objects in the Bergner model structure on simplicially enriched categories; see \cite[\S 16.2]{riehl-cathtpy} for proof.}

\begin{defn}[simplicial computad]\label{defn:simplicial-computad}
A simplicial category $\cA$ is a \emph{simplicial computad} if and only if:
  \begin{itemize}
    \item each category $\cA_n$ of $n$-arrows is freely generated by the graph of atomic $n$-arrows
       \item if $f$ is an atomic $n$-arrow in $\cA_n$ and $\alpha\colon [m]\twoheadrightarrow[n]$ is an epimorphism in $\DDelta$ then the degenerated $m$-arrow $f\cdot\alpha$ is atomic in $\cA_m$.
  \end{itemize}
\end{defn}

\begin{lem} A simplicial category $\cA$ is a simplicial computad if and only if all of its non-identity arrows $f$ can be expressed uniquely as a composite 
\[
  f = (f_1 \cdot \alpha_1) \circ (f_2 \cdot \alpha_2) \circ \cdots \circ (f_\ell \cdot \alpha_\ell)
\]
in which each $f_i$ is non-degenerate and atomic and each $\alpha_i\in\DDelta$ is a degeneracy operator.
\end{lem}
\begin{proof}
This characterization follows immediately from the definition by applying  the Eilenberg-Zilber lemma \cite[II.3.1, pp. 26-27]{GZ},  which says that any degenerate simplex in a simplicial set may be uniquely expressed as a degenerated image of a non-degenerate simplex.
\end{proof}

Free resolutions define simplicial computads, whose atomic $n$-arrows index the generating $n$-homotopies in a homotopy coherent diagram, such as enumerated for the homotopy coherent simplex in \S\ref{sec:shape}.

\begin{prop}\label{prop:resolution-computads} The free resolution $\gC\cA$ is a simplicial computad.
\end{prop}
\begin{proof} Recall $\gC\cA$ is defined to be the free resolution of $\cA$, whose category of $n$-arrows is $(FU)^{n+1}\cA$. The category $FU\cA$ is the free category on the underlying reflexive directed graph of $\cA$. Its arrows are strings of composable non-identity arrows of $\cA$; the atomic 0-arrows are the non-identity arrows of $\cA$.  An $n$-arrow is a string of composable arrows in $\cA$ with each arrow in the string enclosed in exactly $n$ pairs of parentheses. The atomic $n$-arrows are those enclosed in precisely one pair of parentheses on the outside. Since composition in a free category is by concatenation, the unique factorization property is clear.  Since degeneracy arrows ``double up'' on parentheses, these preserve atomics as required.
\end{proof}

We will now do the homework assigned in Exercise \ref{exc:omega}.

\begin{ex}\label{ex:coherent-simplex} Recall the category
 \[
\begin{tikzcd}
\bbomega := & 0 \arrow[r] & 1 \arrow[r] & 2 \arrow[r] & 3 \arrow[r] & \cdots
\end{tikzcd}
\]
The free resolution $\gC\bbomega$ has objects $n \in \bbomega$. 
\begin{itemize}
\item A 0-arrow from $j$ to $k$ is a sequence of non-identity composable morphisms from $j$ to $k$, the data of which is uniquely determined by the objects being passed through. So $0$-arrows from $j$ to $k$ correspond to subsets \[ \{j,k\} \subset T^0 \subset [j,k]\] of the closed internal $[j,k] = \{t \in \bbomega \mid j \leq t \leq k\}$ containing both endpoints.
\item A 1-arrow from $j$ to $k$ is a once bracketed sequence of non-identity composable morphisms from $j$ to $k$. This data is specified by two nested subsets
\[ \{j,k\} \subset T^1 \subset T^0 \subset [j,k]\]
the larger one $T^0$ specifying the underlying unbracketed sequence and the smaller one $T^1$ specifying the placement of the brackets.\footnote{Note the face and degeneracy maps $\begin{tikzcd}[ampersand replacement=\&] (\gC\bbomega)_0 \arrow[r] \& (\gC\bbomega)_1\arrow[l, shift left=0.5em] \arrow[l, shift right=0.5em]  \end{tikzcd}$ are the obvious ones, either duplicating or omitting one of the sets $T^i$.}
\item A $n$-arrow from $j$ to $k$ is an $n$ times bracketed sequence of non-identity composable morphisms from $j$ to $k$, the data of which is specified by nested subsets
\[ \{j,k\} \subset T^n \subset \cdots \subset T^0 \subset [j,k]\] indicating the locations of all of the parentheses.\footnote{The nesting is because parenthezations should be ``well formed'' with open brackets closed in the reverse order to that in which they were opened.}
\end{itemize}

What then are the mapping spaces $\gC\bbomega(j,k)$? When $j > k$ they are empty and when $k=j$ or $k=j+1$ we have $\{j,k\} = [j,k]$ so $\gC\bbomega(j,k) \cong\Delta^0$ is comprised of a single point. For $k > j$, there are $k-j-1$ elements of $[j,k]$ excluding the endpoints and so we see that $\gC\bbomega(j,k)$ has $2^{k-j-1}$ vertices. The $n$-simplices of $\gC\bbomega(j,k)$ are given by specifying $n+1$ vertices --- each a subset $\{j,k\} \subset T^i \subset [j,k]$ ---  that respect the ordering of subsets relation. From this we see that 
\[ \gC\bbomega(j,k) \cong (\Delta^1)^{k-j-1},\footnote{More explicitly, this argument shows that the simplicial set $\gC\bbomega(j,k)$ is the nerve of the poset of subsets $\{j,k\} \subset T \subset [j,k]$ ordered by inclusion.}\] as displayed for instance in the case $j=0$ and $k=4$:
\[ \gC\bbomega(0,4) :=  \left(
\begin{tikzcd}
& {\{0,4\}} \arrow[dl] \arrow[rr] \arrow[dd] \arrow[dr, dotted] \arrow[dddl, dotted] \arrow[ddrr,dotted, bend left=10]  \arrow[dddr, dotted]& &  {\{0,1,4\}} \arrow[dd] \arrow[dl]  \arrow[dddl, dotted] \\ {\{0,3,4\}} \arrow[ddrr, bend left=10, dotted] \arrow[rr, crossing over] \arrow[dd] & & {\{0,1,3,4\}} \\ &  {\{0,2,4\}} \arrow[rr] \arrow[dl] \arrow[dr, dotted] &  & {\{0,1,2,4\}} \arrow[dl] \\ {\{0,2,3,4\}} \arrow[rr] & & {\{0,1,2,3,4\}} \arrow[from=uu, crossing over]
\end{tikzcd}
\right)
\]

The simplicial category $\gC\bbomega$ is a simplicial computad whose atomic $n$-arrows are those with a single outermost parenthezation: i.e., for which $T^n= \{j,k\}$. Geometrically these are all the simplices in the hom space cube $(\Delta^1)^{k-j-1}$ that contain the initial vertex $\{j,k\}$. 
\end{ex}

\subsection{Homotopy coherent realization and the homotopy coherent nerve}\label{sec:nerve}

Employing topological notation, we write $[n] \subset \bbomega$ for the full subcategory spanned by $0,\ldots, n$. 
\[
\begin{tikzcd}
{[n]} := & 0 \arrow[r] & 1 \arrow[r] & 2 \arrow[r] & 3 \arrow[r] & \cdots \arrow[r] & n
\end{tikzcd}
\]
 These categories define the objects of a diagram $\DDelta\hookrightarrow\cat{Cat}$ that is a full embedding: the only functors $[m] \to [n]$ are order-preserving maps from $[m] = \{0,\ldots, m\}$ to $[n] = \{0,\ldots, n\}$. Applying the free resolution construction to these categories we get a functor $\gC \colon \DDelta \to \cat{sCat}$ where $\gC[n]$ is the full simplicial subcategory of $\gC\bbomega$ spanned by those objects $0,\ldots,n$. In particular, its hom spaces are the simplicial cubes described in Example \ref{ex:coherent-simplex}.

\begin{defn}[homotopy coherent realization and nerve]\label{defn:hocoh-adj} The homotopy coherent nerve $\gN$ and homotopy coherent realization $\gC$ are the adjoint pair of functors obtained by applying Kan's construction \cite[1.5.1]{riehl-cathtpy} to the functor
    $\gC\colon \DDelta \to \cat{sCat}$ to construct an adjunction
    \[
\begin{tikzcd}
\cat{sSet} \arrow[r,bend left, "\gC"] \arrow[r, phantom, "\perp"]  & \cat{sCat} \arrow[l, bend left, "\gN"]
\end{tikzcd}
\]

The right adjoint,  called the \textbf{homotopy coherent nerve}, carries a simplicial category $\eS$ to the simplicial set $\gN\eS$ whose $n$-simplices are homotopy coherent diagrams of shape $[n]$ in $\eS$. That is  \[\gN \eS_n := \{ \gC[n] \to \eS\}.\]

The left adjoint is defined by pointwise left Kan extension along the
  Yoneda embedding of \ref{dig:sset}:
\[
\begin{tikzcd}
      {\DDelta}\arrow[rr, hook, "{\Delta^{\bullet}}"] \arrow[dr, "{\gC}"'] &
\arrow[d, phantom, "\cong" ] & {\cat{sSet}}\arrow[dl, dashed, "\gC"] \\
      & {\cat{sCat}} &
    \end{tikzcd}
\]
That is, $\gC\Delta^n$ is defined to be $\gC[n]$ --- a simplicial category that we call the \textbf{homotopy coherent $n$-simplex} --- and for a generic simplicial set $X$, $\gC{X}$ is defined to be a colimit of the homotopy coherent simplices indexed by the category of simplices of $X$.\footnote{Recall from Digression \ref{dig:sset} that the simplicial set $X$ is obtained by gluing in a $\Delta^n$ for each $n$-simplex $\Delta^n \to X$ of $X$. The functor $\gC$ preserves these colimits, so $\gC{X}$ is obtained by gluing in a $\gC[n]$ for each $n$-simplex of $X$.} 

Because of the formal similarity with the geometric realization functor of (\ref{eq:singular-adjunction}), another left adjoint defined by Kan's construction, we refer to $\gC$ as \textbf{homotopy coherent realization}.   
\end{defn}

Left Kan extensions are computed as colimits, providing a formula for the homotopy coherent realization $\gC{X}$ of a simplicial set $X$ as the colimit of a diagram of homotopy coherent simplices $\gC[n]$. However, this does not give very much insight into the mapping space of $\gC{X}$, colimits of simplicial categories being rather complicated. Work of Dugger and Spivak \cite{dugger-spivak}, redeveloped in \cite[\S 4]{RV-VI}, fills this gap; 
see also \cite[\S 16.2-16.4]{riehl-cathtpy}.

\begin{prop}[{\cite[4.4.7]{RV-VI}}]\label{prop:realization-computads} For any simplicial set $X$, its homotopy coherent realization $\gC{X}$ is a simplicial computad in which:
\begin{itemize}
\item objects $\ob\gC{X} = X_0$, the vertices of the simplicial set $X$
\item atomic $0$-arrows are non-degenerate 1-simplices of $X$, the source being the initial vertex and the target being the final vertex of the simplex
\item atomic $1$-arrows are the non-degenerate $k$-simplices of $X$ for $k > 1$, the source being the initial vertex and the target being the final vertex of the simplex
\item atomic $n$-arrows are pairs comprised of a non-degenerate $k$-simplex in $X$ for some $k > n$ together with a set of proper inclusions
\begin{equation}
\label{eq:flag} \{0,k\} = T^n \subsetneq T^{n-1} \subsetneq \cdots \subsetneq T^1 \subsetneq T^0 = [0,k]\end{equation}
the data of which defines an atomic $n$-arrow in $\gC\Delta^k$ from $0$ to $k$ that is not in the image of any of the face maps. This source of this $n$-arrow is the initial vertex of the $k$-simplex, while the target is the final vertex of the simplex.
\end{itemize}
\end{prop}

Note that the description of atomic $n$-arrows subsumes those of the atomic 0-arrows and atomic 1-arrows. 
The data of a non-degenerate atomic $n$-arrow from $x$ to $y$ in $\gC{X}$ is given by a ``bead,'' that is a non-degenerate $k$-simplex in $X$ from $x$ to $y$, together with the additional data of a sequence of proper subset inclusions (\ref{eq:flag}), which Dugger and Spivak refer to as a ``flag of vertex data.'' 
Non-atomic $n$-arrows are then ``necklaces,'' that is strings of beads in $X$ joined head to tail, together with accompanying ``vertex data'' for each simplex.

\begin{proof}[Proof sketch]
This can be proven inductively using the skeletal decomposition of the simplicial set $X$ and will reveal that for any inclusion of simplicial sets $X \subset Y$, the functor $\gC{X} \hookrightarrow\gC{Y}$ of homotopy coherent realizations is a \textbf{simplicial subcomputad inclusion}: a functor of simplicial computads that is injective on objects and faithful and also preserves atomic arrows. This is proven by verifying directly that $\gC{\partial\Delta^k}\hookrightarrow\gC{\Delta^k}$ is a simplicial subcomputad inclusion and then arguing that such inclusions are closed under coproduct, pushout, and transfinite composition in simplicial categories. It follows that $\gC{X}$ is a simplicial computad, and moreover the analysis of what happens when attaching an $k$-simplex provides the description of atomic $n$-arrows given above. The data (\ref{eq:flag}) represents an $n$-arrow in $\gC\Delta^k(0,k)$ that is atomic (since $T^n=\{0,k\}$) and not contained in any face (since $T^0 = [0,k]$).
\end{proof}

Applying Kan's construction to the embedding $\DDelta \hookrightarrow\cat{Cat}$ of the ordinal categories  
yields an adjunction
    \[
\begin{tikzcd}
\cat{sSet} \arrow[r,bend left, "\Ho"] \arrow[r, phantom, "\perp"]  & \cat{Cat} \arrow[l, bend left, hook']
\end{tikzcd}
\]
  the right adjoint of which is called the \textbf{nerve} and
  the left adjoint of which, defined by pointwise left Kan extension along the
  Yoneda embedding:
\[
\begin{tikzcd}
      {\DDelta}\arrow[rr, hook, "{\Delta^{\bullet}}"] \arrow[dr, hook] &
\arrow[d, phantom, "\cong" ] & {\cat{sSet}}\arrow[dl, dashed, "\Ho"] \\
      & {\cat{Cat}} &
    \end{tikzcd}
\]
defines the homotopy category of a simplicial set, via a mild generalization of the construction introduced for quasi-categories at the end of Part \ref{part:I}. For a category $\cA$, an $n$-simplex in the nerve of $\cA$ is simply a functor $[n] \to \cA$, i.e., a string of $n$-composable morphisms in $\cA$. Note that by fullness of $\DDelta\hookrightarrow\cat{Cat}$, the nerve of the ordinal category $[n]$ is the standard $n$-simplex $\Delta^n$.

The nerve functor defines a fully faithful embedding $\cat{Cat}\hookrightarrow\cat{sSet}$ of categories into simplicial sets that lands in the subcategory spanned by the quasi-categories. In quasi-category theory, it is  very convenient to conflate a category with its nerve, which is why we have not introduced notation for this right adjoint.\footnote{Note that the nerve of the category $\cat{B}G$ with a single object and elements of the group $G$ as its endomorphisms is the Kan complex that typically goes by this name.}

With this convention, we now have two simplicial categories we have denoted $\gC\cA$ for a small category $\cA$:  the free resolution of $\cA$ and the homotopy coherent realization of the nerve of $\cA$. This would be confusing were these objects not naturally isomorphic:\footnote{Corollary \ref{cor:hocoh-diagram-adjunction}, which implies Theorem \ref{thm:resolution-equals-realization}, is stated without proof in \cite{CP-vogt}. The argument given here appears in \cite{riehl-necklace}, though it is highly probable that an earlier proof exists in the literature.}

\begin{thm}[{\cite[6.7]{riehl-necklace}}]\label{thm:resolution-equals-realization} For any category $\cA$, $\gC \cA \cong \gC\cA$. That is, its free resolution is naturally isomorphic to the homotopy coherent realization of its nerve.
\end{thm}

\begin{rmk}
Note $\gC\Delta^n \cong \gC[n]$ is tautologous. The left Kan extension along the Yoneda embedding is defined so as to agree with $\gC \colon\DDelta \to \cat{sCat}$ on the subcategory of representables. Many arguments involving simplicial sets can be reduced to a check on representables, with the extension to the general case following formally by ``taking colimits.'' This result, however,  is not one of them since we are trying to prove something for all \emph{categories} and the embedding $\cat{Cat}\hookrightarrow\cat{sSet}$ does not preserve colimits.
\end{rmk}

\begin{proof}
Proposition \ref{prop:resolution-computads} and Proposition \ref{prop:realization-computads} reveal that both simplicial categories are simplicial computads. We will argue that they have the same objects and non-degenerate atomic $n$-arrows.

Both have $\ob\cA$ as objects, these being the vertices in the nerve of $\cA$. Atomic 0-arrows of the free resolution are morphisms in $\cA$; while atomic 0-arrows in the coherent realization are non-degenerate 1-simplices of the nerve --- these are the same thing. Atomic non-degenerate 1-arrows of the free resolution are sequences of at least two morphisms (enclosed in a single set of outer parentheses), while atomic 1-arrows of the coherent realization are non-degenerate simplices of dimension at least two --- again these are the same. Finally a non-degenerate atomic $n$-arrow is a sequence of $k$ composable morphisms with $(n-1)$ non-repeating bracketings; this non-degeneracy necessitates $k > n$. This data defines a $k$-simplex in the nerve together with a non-degenerate atomic $n$-arrow in $\gC[k](0,k)$, i.e., an atomic $n$-arrow in the coherent realization.
\end{proof}

Because the homotopy coherent realization was defined to be the left adjoint to the homotopy coherent nerve, it follows immediately:

\begin{cor}\label{cor:hocoh-diagram-adjunction} $\quad$
\begin{enumerate}
\item Homotopy coherent diagrams of shape $\cA$ in a simplicial category $\eS$ correspond to maps from the nerve of $\cA$ to the homotopy coherent nerve of $\eS$: i.e., there is a natural bijection between simplicial functors $\gC \cA \to \eS$  and simplicial maps $\cA \to \gN\eS$. 
\item Hence, the simplicial set $\fun{Coh}(\cA,\eS)$ is isomorphic to the simplicial set $\gN\eS^\cA$ defined using the internal hom in $\sSet$.
\end{enumerate}
\end{cor}

Note that the homotopy coherent realization functor is defined on all simplicial sets $X$. Extending previous terminology, we refer to a simplicial functor $\gC{X} \to \eS$ as a \textbf{homotopy coherent diagram of shape $X$ in $\eS$}.

The simplicial computad structure of Proposition \ref{prop:realization-computads} can also be used to prove the following important result:

\begin{thm}[{\cite[2.1]{CP-vogt}}]\label{thm:nerve-to-quasi}
 If $\eS$ is Kan complex enriched, then $\gN\eS$ is a quasi-category.
 \end{thm}
\begin{proof}
By adjunction, to extend along an inner horn inclusion $\Lambda^n_k\hookrightarrow\Delta^n$ mapping into the homotopy coherent nerve $\gN\eS$ is to  extend along simplicial subcomputad inclusions $\gC\Lambda^n_k\hookrightarrow\gC\Delta^n$ mapping into the Kan complex enriched category $\eS$. This is the simplicial subcomputad generated by all arrows whose beads are supported by simplices in $\Lambda^n_k\subset\Delta^n$. The only missing ones are in the mapping space from $0$ to $n$, so we are asked to solve a single lifting problem
\[
\begin{tikzcd}
\gC\Lambda^n_k(0,n) \arrow[r] \arrow[d, hook] & \Map(X_0,X_n) \\ \gC\Delta^n(0,n) \arrow[ur, dashed]
\end{tikzcd}
\] 
In Example \ref{ex:coherent-simplex}, we have seen that $\gC\Delta^n(0,n) \cong (\Delta^1)^{n-1}$ is a cube. One can similarly check that $\gC\Lambda^n_k(0,n)$ is a cubical horn. Cubical horn inclusions can be filled in the Kan complex $\Map(X_0,X_n)$, completing the proof.
\end{proof}

\begin{cor}\label{cor:coherent-diagram-quasi-category}
$\fun{Coh}(\cA,\eS) \cong \gN\eS^\cA$  is a quasi-category.
\end{cor}
\begin{proof}
By the adjunction of Definition \ref{defn:hocoh-adj}, a simplicial functor $\gC \cA \to \eS$ is the same as a simplicial map $\cA \to \gN\eS$. So $\fun{Coh}(\cA,\eS) \cong \gN\eS^\cA$ and since the quasi-categories define an exponential ideal in simplicial sets, the fact that $\gN\eS$ is a quasi-category implies that $\gN\eS^\cA$ is too.
\end{proof}

\begin{rmk}[all diagrams in homotopy coherent nerves are homotopy coherent]
This corollary explains that any map of simplicial sets $X \to \gN\eS$ transposes to define a simplicial functor $\gC{X} \to \eS$, a homotopy coherent diagram of shape $X$ in $\eS$. While not every quasi-category is isomorphic to a homotopy coherent nerve of a Kan complex enriched category, a deep theorem shows that every quasi-category is equivalent to a homotopy coherent nerve; one proof appears as \cite[7.2.22]{RV-VI}. This explains the slogan introduced at the beginning of this part, that all diagrams in quasi-categories are homotopy coherent.
\end{rmk}


\subsection{Further applications}

Using similar techniques, where a homotopy coherent realization problem is transposed along the adjunction $\gC \dashv \gN$ to an extension problem in simplicial sets mapping into the homotopy coherent nerve, one can show:

\begin{prop}[{\cite[\S 2]{CP-maps}}] Given a homotopy coherent diagram $F \colon \gC\cA \to \eS$ in a locally Kan simplicial category and a family of homotopy equivalences $f_a \colon Fa \to Ga$ for all $a \in \cA$, there is a homotopy coherent diagram $G$ and coherent map $f \colon F \to G$ that moreover defines an isomorphism in $\cat{Coh}(\cA,\eS)$.
\end{prop}
\begin{proof} By colimits reduce to the case $\cA=\Delta^n$ and construct the desired extension
\[
\begin{tikzcd}
\Delta^n \times \Lambda^1_1 \cup \sk_0\Delta^n \times \Delta^1 \arrow[d, tail] \arrow[r] & \gN\eS \\ \Delta^n \times \Delta^1 \arrow[ur, dashed]
\end{tikzcd}
\]
See \cite[\S 2]{CP-maps} for a very explicit description of what this filling process looks like in low dimensions. 
\end{proof}

\begin{prop}[{\cite[\S 3]{CP-maps}}] Given a homotopy coherent map $f \colon F \to G$ of homotopy coherent diagrams $F,G \colon \gC\cA \to \eS$ in a locally Kan simplicial category and a family of homotopies $f_a \simeq g_a \colon Fa \to Ga$ for all $a \in \cA$, there is a homotopy coherent map $g \colon F \to G$ extending these component maps together with a coherent homotopy of homotopy coherent maps $H \colon \gC(\cA\times [2]) \to \eS$.
\end{prop}
\begin{proof}
The argument is analogous to the previous inductive one solving the lifting problems: 
\[
\begin{tikzcd}
\Delta^n \times \Lambda^2_0 \cup \sk_0\Delta^n \times \Delta^2 \arrow[d, tail] \arrow[r] & \gN\eS \\ \Delta^n \times \Delta^2 \arrow[ur, dashed]
\end{tikzcd} 
\]
\end{proof}

\section{Homotopy coherent algebra}\label{part:III}

In \S\ref{sec:nerve}, we saw that the nerve of a 1-category $\cA$ defines the shape of a homotopy coherent diagram taking values in a quasi-category. In this final part, we will extend this principle, to argue that if $\mathbb{A}$ is a 2-category, then its ``local nerve'' (taking the nerve of the hom-categories to produce a simplicial category) defines the shape of homotopy coherent categorical structure in a quasi-categorically enriched category --- at least in the case where the local nerve of $\mathbb{A}$ happens to be a simplicial computad. We will pursue this line of thought in two particular examples, where $\mathbb{A}$ indexes a \emph{monad} or an \emph{adjunction}.

\subsection{From coherent homotopy theory to coherent category theory}

We have observed that categories of ``space-like'' objects are frequently enriched over \emph{Kan complexes}, simplicial sets in which any horn can be filled to a simplex. The 0-arrows of the mapping objects $\Map(X,Y)$ in such a category $\eS$ define morphisms $X \to Y$ in the underlying category of $\cS$. The Kan complex property implies that all $n$-arrows for $n > 0$ are ``invertible'' in a suitable sense, so we interpret an $n$-arrow in a mapping Kan complex $\Map(X,Y)$ as an $n$-homotopy, with the case $n=1$ defining ordinary homotopies between parallel morphisms $f,g \colon X \to Y$. The totality of the data of $\eS$ may be thought of as defining an $(\infty,1)$-\emph{category}, meaning a category with objects and morphisms in each dimension, all but the lowest of which are invertible.

There is another context in which we are used to having multiple dimensions of morphisms, namely category theory itself. Famously, the category $\cat{Cat}$ of ordinary categories and functors is a \emph{2-category}. Here the morphisms between morphisms are the natural transformations of Definition \ref{defn:nat-trans}.  As mentioned there, natural transformations are analogous to homotopies in the sense that they can be expressed as functors $H \colon \cC \times [1] \to \cD$ defined using the interval category $[1] = 0 \to 1$, but unlike homotopies natural transformations are not typically invertible, an important amount of extra flexibility.

Thus, the appropriate context for homotopy coherent category theory will be a category that is simplicially enriched but with two non-invertible dimensions of morphisms rather than just one. More precisely, a categorical context for homotopy coherent category theory is a simplicial category $\KK$ that is \emph{quasi-categorically enriched} as opposed to Kan complex enriched, in which case it is traditional to write $\Fun(X,Y)$ for the function complexes instead of $\Map(X,Y)$.

\begin{ex}\label{ex:inf-cats} The categories of quasi-categories, Segal categories, complete Segal spaces, and naturally marked simplicial sets (1-complicial sets) are all enriched over quasi-categories. 
\end{ex}

\begin{ex}\label{ex:fib-inf-cats}
Suitably defined categories of fibrations  (isofibrations, cartesian fibrations, cocartesian fibrations) of any of these over a fixed base are also enriched over quasi-categories.
\end{ex}

In each category just mentioned, the objects are a model of an $(\infty,1)$-category --- which, in deference to Lurie, most people call $\infty$-\emph{categories} --- or a fibered variant of the above. So if we develop homotopy coherent category theory in the context of any category enriched over quasi-categories, we are doing ``model independent $\infty$-category'' in a rather strict framework.

\begin{dig}[on model independent $\infty$-category theory]
Some people use ``model independent $\infty$-category theory'' to refer either to some sort of hand-wavy $\infty$-category theory or to something that is secretly quasi-category theory but presented in somewhat different language. The idea in both cases is that an expert could make everything precise. As an aesthetic and expository philosophy, this approach makes a lot of sense, but my concern at present is that there may be too few ``experts.''

A more conservative deployment of model independent $\infty$-category theory refers to mathematics that can be:
\begin{itemize}
\item specialized to the case of quasi-categories and so recover a theory that is compatible with the theory of Joyal and Lurie and
\item specialized to other models of $(\infty,1)$-categories and recover something equivalent to this in the sense that it is preserved and reflected by ``change of universe functors.''\footnote{The right Quillen equivalences between quasi-categories, Segal categories, complete Segal spaces, and 1-complicial sets established by Joyal and Tierney \cite{JT} all define \emph{biequivalences} of quasi-categorically enriched categories: functors that are surjective on objects up to equivalence and define a local equivalence of quasi-categories.}
\end{itemize}
This is the sense in which ``model independent $\infty$-category'' will be used here, referring to constructions and theorems that can be used for a variety of models of $(\infty,1)$-categories in exactly the same way in each instance and without relying upon any details of the models, except to know that each category of $\infty$-categories is enriched of quasi-categories.\footnote{For a considerably more detailed account of what it means to develop $\infty$-category theory ``model independently'' from this point of view see \cite{riehl-scratch, RV-elements}.}
\end{dig}

A final example of a quasi-categorically enriched category is worth mentioning:

\begin{ex}\label{ex:cat-cat} The category $\Cat$ of categories is self-enriched: for any pair of small category $\cC$ and $\cD$, we may define the category $\cD^\cC$ of functors and natural transformations. Passing to the nerve, this defines a quasi-categorically enriched category of categories, since nerves of ordinary categories are quasi-categories.
\end{ex}

\subsection{Homotopy coherent monads}

A \textbf{monad} on a category $\cB$ is a syntactic way of encoding ``algebraic structure'' that might be borne by objects in $\cB$. Various other mechanisms for describing finitary algebraic operations satisfying equations exist --- for instance operads or Lawvere theories --- but monads are able to capture more general varieties of algebraic structure.

\begin{defn}\label{defn:monad} A \textbf{monad} on $\cB$ is given by an endofunctor $T \colon \cB \to \cB$ together with a pair of natural transformations $\eta \colon \id_\cB \to T$ and $\mu \colon T^2 \to T$ so that the following ``associativity'' and ``unit'' diagrams commute.
\[
\begin{tikzcd}
T^3 \arrow[r, "T\mu"] \arrow[d, "\mu T"'] & T^2 \arrow[d, "\mu"] & T \arrow[r, "T\eta"] \arrow[dr, equals] & T^2 \arrow[d, "\mu"] & T \arrow[l, "\eta T"'] \arrow[dl, equals] \\ T^2 \arrow[r, "\mu"'] & T & & T
\end{tikzcd}
\]
\end{defn}

\begin{ex}\label{ex:mnd-from-operad} Let $\{ P_n\}_{n \in \NN}$ be a symmetric operad in sets. Then if $\cB$ has finite products and the colimits displayed below, the associated monad is defined for $b \in \cB$ by
\[ Tb := \sum_{n \in \NN} P_n \times_{\Sigma_n} b^n.\]
\end{ex}

We will review the notion of an \textbf{algebra} for a monad $T$ on $\cB$ in Definition \ref{defn:algebra}. For now,  it suffices to mention that the algebras for the monad construction in Example \ref{ex:mnd-from-operad} are equivalent to the algebras for the operad.

A monad, as just defined, is a diagram inside $\Cat$ whose image is comprised of 
\begin{itemize}
\item a single object, the category $\cB$ on which the monad acts, 
\item a $0$-arrow $T$, the monad endofunctor,
\item a pair of 1-arrows $\eta \colon \id_B \to T$ and $\mu \colon T^2 \to T$, the monad natural transformations,
\end{itemize}
satisfying the axioms of Definition \ref{defn:monad}.

Let us try to naively conjure the data of a \emph{homotopy coherent monad} before stating the full definition. That is, let us try to define a simplicial computad $\Mnd$ so that simplicial functors $\Mnd \to\KK$ valued in a quasi-categorically enriched category define a monad on an object in $\KK$.

Firstly, $\Mnd$ should have a single object, which we denote by $+$, whose image identifies the  object of $\KK$ on which the monad acts. Since $\Mnd$ has a single object we only need to describe the simplicial set $\Mnd(+,+)$ of endo-arrows. 

There is a single generating $0$-arrow $t$, whose image defines the endofunctor of the monad. Then the $0$-arrows are $t^n$ for all $n \geq 0$ with $t^0 = \id_+$.

Among the generating 1-arrows we should have $\eta \colon \id_+ \to t$ and $\mu \colon t^2 \to t$. But our intuition is that to define a ``homotopy coherent'' algebraic structure, we should avoid making unnecessary choices. This suggests that it would be better to have a generating $1$-arrow $\mu_n \colon t^n \to t$ for all $n \geq 0$, the $n$-ary multiplication map, where the case $\mu_1 = \id_t$ and $\mu_0$ we think of as the unit $\eta$. 

By ``horizontally'' composing the atomic $1$-arrows $\mu_{i_1},\ldots, \mu_{i_m}$ we obtain composite 1-arrows $t^{\sum_j i_j} \to t^m$ defined as follows:
\[
\begin{tikzcd}
{+} \arrow[r, bend left, "t^{i_1}"] \arrow[r, phantom, "\scriptstyle{\downarrow \mu_{i_1}}"] \arrow[r, bend right, "t"'] & {+} \arrow[r, bend left, "t^{i_2}"] \arrow[r, phantom, "\scriptstyle{\downarrow \mu_{i_2}}"] \arrow[r, bend right, "t"'] & {+} \arrow[r, phantom, "\cdots"]  &{+} \arrow[r, bend left, "t^{i_m}"] \arrow[r, phantom, "\scriptstyle{\downarrow \mu_{i_m}}"] \arrow[r, bend right, "t"'] & {+}
\end{tikzcd}
\]
Because each of the generating 1-arrows $\mu_n$ has codomain $t$, these composite $1$-arrows are uniquely determined by interpreting  the codomain $t^m$ as $m$ copies of $t$ each of which receives some map $\mu_i$. The data of the map $\mu[\alpha] \colon t^n \to t^m$ is then given by an order-preserving map $\alpha \colon \{0,\ldots, n-1\} \to \{0,\ldots, m-1\}$, which can be thought of as a specification of the cardinality of the fiber over each $j \in \{0,\ldots, m-1\}$. This gives a complete description of the 1-arrows in $\Mnd(+,+)$.

What 2-arrows should there be in $\Mnd(+,+)$? Associativity says that ``all maps from $t^n$ to $t$ should agree.''. In a homotopy coherent context, relations become data witnessed by arrows of the next dimension up. This suggests that for any $n,m \geq 0$ and any simplicial map $\alpha \colon \{0,\ldots, n-1\} \to \{0,\ldots, m-1\}$ we should have a 2-arrow with boundary
\[
\begin{tikzcd} & t^m \arrow[dr, "\mu_m"] \\ t^n \arrow[ur, "{\mu[\alpha]}"] \arrow[rr, "\mu_n"'] & & t 
\end{tikzcd}
\]
None of these relations implies the other so they should all be generating. The composite 2-arrows are then of the form
\[
\begin{tikzcd} & t^m \arrow[dr, "{\mu[\beta]}"] \\ t^n \arrow[ur, "{\mu[\alpha]}"] \arrow[rr, "{\mu[\gamma]}"'] & & t^k 
\end{tikzcd}
\]
for $k >1$ whenever $\gamma$ and $\beta \alpha$ define the same function $[n-1] \to [k-1]$. A similar description can be given for the $m$-arrows for $m \geq 3$.

So what is $\Mnd(+,+)$? It has 0-arrows indexed by natural numbers $n \geq 0$, 1-arrows corresponding to all order preserving functions, 2-arrows corresponding to composable pairs of order preserving functions, etc. So we see that $\Mnd(+,+)$ is isomorphic to the nerve of the category $\DDelta_+$  of finite ordinals and order-preserving maps.

\begin{defn}[the free homotopy coherent monad] Let $\Mnd$ denote the quasi-cat\-e\-gor\-ically enriched category with a single object $+$ and whose endo-hom quasi-category 
\[ \Mnd(+,+) := \DDelta_+\] is the nerve of the category of finite ordinals and order-preserving maps. Composition is given by the ordinal sum
\[ \begin{tikzcd}\Mnd(+,+) \times \Mnd(+,+) \arrow[r, "\oplus"] & \Mnd(+,+).\end{tikzcd}\]
\end{defn}

Ignoring the nerve, we can think of $\Mnd$ as a strict 2-category. It has a universal property that is well-known:

\begin{prop}\label{prop:monad-UP} 2-functors $\Mnd \to \Cat$ correspond to monads.
\end{prop}
\begin{proof}
A 2-functor $\Mnd \to \Cat$ picks out a category $\cB$ as the image of $+$, and then defines a strictly monoidal functor $\Mnd(+,+)=\DDelta_+ \to \cB^\cB$. The category $\DDelta_+$ has a universal property: strictly monoidal functors out of $\DDelta_+$ correspond to monoids in the target, and a monad on $\cB$ is just a monoid in the category $\cB^\cB$ of endofunctors!
\end{proof}

Thus reassured, we may define a notion of a homotopy coherent monad acting on an object in a quasi-categorically enriched category $\KK$. On account of the examples listed in \ref{ex:inf-cats}, \ref{ex:fib-inf-cats}, and \ref{ex:cat-cat}, we might think of the objects in a quasi-categorically enriched category as being ``$\infty$-categories'' in some sense.

\begin{defn}[{\cite{RV-II}}]\label{defn:hocoh-monad} A \textbf{homotopy coherent monad} in a quasi-cat\-e\-gor\-ically enriched category $\KK$ is a simplicial functor $\Mnd \to \KK$ whose domain is the simplicial computad $\Mnd$. Explicitly, it picks out:
\begin{itemize}
\item an object $B \in \KK$.
\item a homotopy coherent diagram $\DDelta_+ \to \Fun(B,B)$ that is strictly monoidal with respect to composition. It sends the generating 0-arrow $t\colon + \to +$ to a 0-arrow $T \colon B \to B$ and identifies 1-arrows that assemble into a diagram
\[
\begin{tikzcd}
\id_B \arrow[r, "\eta" description] & T \arrow[r, shift left=0.75em, "\eta T"] \arrow[r, shift right=0.75em, "T\eta"'] & T^2 \arrow[l, "\mu" description] \arrow[r] \arrow[r, shift left=1.5em] \arrow[r, shift right=1.5em] & T^3 \arrow[l, shift left=0.75em] \arrow[l, shift right=0.75em] & \cdots
\end{tikzcd}
\]
\end{itemize}
\end{defn}

We interpret the simplicial functor $\DDelta_+ \to \Fun(B,B)$ defined by a homotopy coherent monad as being a homotopy coherent version of the monad resolution for $(T,\eta,\mu)$.

\subsection{Homotopy coherent adjunctions}

This definition of a homotopy coherent monad seems reasonable but are there any examples? One is given by the homotopy coherent monad on the large quasi-category $\prod_{\ob B}\qCat$ whose algebras are those $\ob{B}$-indexed families of quasi-categories that assemble into the fibers for a cartesian fibration over the quasi-category $B$ \cite{RV-IX}. This example arises from a familiar path. In traditional category theory, all monads arise from adjunctions. After reviewing the classical definitions, we will observe that up to homotopy adjunctions can be extended to homotopy coherent adjunctions, as defined below. Homotopy coherent adjunctions then restrict to define homotopy coherent monads. Moreover, monadic homotopy coherent adjunctions also recover the algebras for a homotopy coherent monad, to be introduced in \S\ref{sec:algebras}.

\begin{defn}\label{defn:adjunction} An \textbf{adjunction} in $\Cat$ is comprised of a pair of categories $\cA$ and $\cB$ together with functors $U\colon \cA \to \cB$ and $F \colon \cB \to \cA$ and natural transformations $\eta \colon \id_\cB \to UF$ and $\epsilon \colon FU \to \id_\cA$, called the \textbf{unit} and \textbf{counit} respectively, so that the diagrams
\begin{equation}\label{eq:triangles}
\begin{tikzcd}
& FUF\arrow[dr, "\epsilon F"] & & & UFU \arrow[dr, "U\epsilon"] \\ F \arrow[ur, "F\eta"] \arrow[rr, equals] & & F & U \arrow[ur, "\eta U"] \arrow[rr, equals] & & U
\end{tikzcd}
\end{equation}
commute.
\end{defn}

\begin{lem} Any adjunction induces a monad $(UF,\eta, U\epsilon F)$ on $\cB$.
\end{lem}
\begin{proof}
An exercise in diagram chasing.
\end{proof}

There is a free-living 2-category $\Adj$ containing an adjunction in the sense of a universal property analogous to Proposition \ref{prop:monad-UP}. It has two objects $+$ and $-$ and the four hom-categories displayed:
\[
\begin{tikzcd}[column sep=large]
+ \arrow[r, bend left, "\DDelta_\bot\cong\DDelta_\top^\op"] \arrow[loop left, "\DDelta_+"] \arrow[r, phantom, "\perp"] & - \arrow[loop right, "\DDelta_+^\op"] \arrow[l, bend left, "\DDelta_\top\cong\DDelta_\bot^\op"]
\end{tikzcd}
\]
Here $\DDelta_\top, \DDelta_\bot \subset \DDelta \subset \DDelta_+$ are the subcategories of order-preserving maps that preserve the top or bottom elements, respectively, in each ordinal. The composition maps in $\Adj$ are all restrictions of the ordinal sum operation.

\begin{prop}[{\cite{SS}}] 2-functors $\Adj \to \Cat$ correspond to adjunctions in $\Cat$.
\end{prop}

We saw in Definition \ref{defn:hocoh-monad} that the free homotopy coherent monad $\Mnd$ is in fact the free 2-category containing a monad: when this 2-category is regarded as a simplicial category by identifying its hom-category $\DDelta_+$ with its nerve, this simplicial category turns out to be a simplicial computad whose atomic $n$-arrows are those $n$-simplices whose final vertex is the 0-arrow $t$. The following result tells us that the same is true for adjunctions: the 2-category $\Adj$, when regarded as a simplicial category, is a simplicial computad that defines the \textbf{free homotopy coherent adjunction}. Moreover, we present a convenient graphical description of its $n$-arrows that establishes this simplicial computad structure.

\begin{prop}[{\cite{RV-II}}]  The 2-category $\Adj$, when regarded as a simplicial category via the nerve, is a simplicial computad with:
\begin{itemize}
\item two objects $+$ and $-$
\item two atomic $0$-arrows $f \colon + \to -$ and $u \colon - \to +$
\item $n$-arrows given by strictly undulating squiggles on $(n+1)$-lines that start and end in the regions labelled ``$-$'' or ``$+$''
\begin{center}
{\includegraphics[width=1.5in]{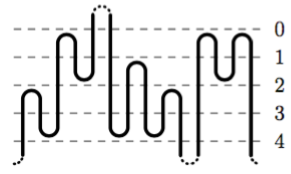}}
\end{center}
that are atomic if and only if there are no ``instances of $+$ or $-$'' in their interiors.
\end{itemize}
\end{prop}
\begin{proof} An $n$-arrow lies in $\Adj(-,+)$ if starts in the space labeled $-$ on the right and ends in the space labeled + on the left; the description of the other hom simplicial sets is similar. The face and degeneracy maps act on the simplicial sets $\Adj(+,+)$, $\Adj(-,+)$, $\Adj(+,-)$, and $\Adj(+,+)$ by removing and duplicating lines. Composition is by horizontal juxtaposition, which makes the simplicial computad structure clear.
\end{proof}

\begin{rmk} Note this gives a graphical calculus on the full subcategory $\Mnd\hookrightarrow\Adj$. The $n$-arrows are strictly undulating squiggles on $(n+1)$-lines that start and end at the space labeled $+$; these are atomic if and only if there are no instances of $+$ in their interiors. This condition implies that if all the lines are removed except the bottom one, a process that computes the final vertex of the $n$-simplex, the resulting squiggle looks like a single hump over one line, which is the graphical representation of the 0-arrow $t$. Because atomic arrows in $\Mnd$ may pass through $-$,  $\Mnd\hookrightarrow\Adj$ is not a subcomputad inclusion.
\end{rmk}

\begin{defn} A \textbf{homotopy coherent adjunction} in a quasi-categorically enriched category $\KK$ is a simplicial functor $\Adj \to \KK$. Explicitly, it picks out:
\begin{itemize}
\item a pair of objects $A, B \in \KK$.
\item homotopy coherent diagrams \[\DDelta_+ \to \Fun(B,B),\quad \DDelta_+^\op \to \Fun(A,A),\quad \DDelta_\top \to \Fun(A,B),\quad \DDelta_\top^\op \to \Fun(B,A)\] that are functorial with respect to the composition action of $\Adj$. 
\end{itemize}
The 0- and 1-dimensional data of the first and third of these may be depicted as follows
\[
\begin{tikzcd}[column sep=large]
\id_B \arrow[r, "\eta" description] & UF \arrow[r, shift left=0.75em, "\eta UF"] \arrow[r, shift right=0.75em, "UF\eta"'] & UFUF \arrow[l, "U\epsilon F" description] \arrow[r] \arrow[r, shift left=1.5em] \arrow[r, shift right=1.5em] & UFUFUF \arrow[l, shift left=0.75em] \arrow[l, shift right=0.75em] & \cdots
\end{tikzcd}
\]
\[
\begin{tikzcd}[column sep=large]
U \arrow[r, shift left=0.25em, "\eta U"] & UFU \arrow[l, shift left=0.25em, "U \epsilon"] \arrow[r, shift right=0.5em, "UF\eta" description] \arrow[r, shift left=1.25em, "\eta UF"] & UFUFU \arrow[l, shift right=0.5em, "U\epsilon F" description] \arrow[l, shift left=1.25em, "UFU\epsilon"]    \arrow[r, shift left=0.3em] \arrow[r, shift left=1.5em] \arrow[r, shift right=0.9em] & UFUFUFU   \arrow[l, shift left=0.3em] \arrow[l, shift left=1.5em] \arrow[l, shift right=0.9em]  & \cdots
\end{tikzcd}
\]
the remaining two diagrams being dual. We interpret the homotopy coherent diagrams $\DDelta_+ \to \Fun(B,B)$, $\DDelta_+^\op \to \Fun(A,A)$, $\DDelta_\top \to \Fun(A,B)$, and $\DDelta_\top^\op \to \Fun(B,A)$ as
defining homotopy coherent versions of the bar and cobar resolutions of the adjunction $(F,U, \eta,\epsilon)$.
\end{defn}

Any homotopy coherent adjunction has an underlying adjunction in the following sense. A quasi-categorically enriched category $\KK$ has an associated \textbf{homotopy 2-category} defined by applying the homotopy category functor $\Ho$ to each hom quasi-category. Now, an \textbf{adjunction in a quasi-categorically enriched category} is an adjunction the homotopy 2-category obtained by taking the hom-categories of the function complexes. Explicitly, an adjunction is given by:
\begin{itemize}
\item a pair of objects $A$ and $B$,
\item a pair of 0-arrows $U \colon A \to B$ and $F \colon B \to A$,
\item a pair of 1-arrows $\eta \colon \id_B \to UF$ and $\epsilon \colon FU \to \id_A$
\end{itemize}
so that there exist 2-arrows whose boundaries have the form displayed in (\ref{eq:triangles}).

The upshot is that an adjunction in a quasi-categorically enriched category is not so hard to define in practice and this low-dimensional data may be extended to give a full homotopy coherent adjunction:

\begin{thm}[{\cite[4.3.11,4.3.13]{RV-II}}]\label{thm:coh-adj} Any adjunction in a quasi-categorically enriched category extends to a homotopy coherent adjunction.
\end{thm}

Moreover extensions from judiciously chosen basic adjunction data are homotopically unique \cite[\S 4.4]{RV-II}.

\begin{rmk} It is also fruitful to consider simplicial functors $\Adj \to \eS$ valued in a Kan complex enriched category. Because all 1-arrows in $\eS$ are invertible, the unit and counit in this case are natural isomorphisms and this data is more properly referred to as a \textbf{homotopy coherent adjoint equivalence}. Theorem \ref{thm:coh-adj} implies that any adjoint equivalence in a Kan complex enriched category extends to a homotopy coherent adjoint equivalence. Paired with the familiar 2-categorical result that says that any equivalence can be promoted to an adjoint equivalence, we conclude that any equivalence in a Kan complex enriched category extends to a homotopy coherent adjoint equivalence.
\end{rmk}

\subsection{Algebras for a homotopy coherent monad}\label{sec:algebras}

Finally, we connect these homotopy coherent notions to ``algebra.''

\begin{defn}\label{defn:algebra} Let $(T,\eta,\mu)$ be a monad acting on a category $\cB$. A $T$-\textbf{algebra} is a pair $(b, \beta \colon Tb\to b)$ so that 
\[
\begin{tikzcd}
b \arrow[r, shift left=0.25em, "\eta"] & Tb \arrow[l, shift left=0.25em, "\beta"] \arrow[r, shift right=0.5em, "T\eta" description] \arrow[r, shift left=1.25em, "\eta T"] & T^2b \arrow[l, shift right=0.5em, "\mu" description] \arrow[l, shift left=1.25em, "T\beta"]
\end{tikzcd}
\]
defines a truncated split augmented simplicial object.\footnote{The shape of this diagram is given by the full subcategory of $\DDelta_\top$ spanned by the objects $[0]$, $[1]$, and $[2]$.}
\end{defn}

$T$-algebras in $\cB$ and $T$-algebra homomorphisms define the \textbf{category of algebras}, traditionally denoted by $\cB^T$.

\begin{prop} Let $(T,\eta,\mu)$ be a monad acting on a category $\cB$. There is an adjunction
\[
\begin{tikzcd}
\cB \arrow[r, bend left, "F^T" pos=0.6] \arrow[r, phantom, "\perp"] & \cB^T \arrow[l, bend left, "U^T" pos=0.4]
\end{tikzcd}
\]
whose underlying monad is $T$.
\end{prop}

If $B \in \KK$ is thought of as an $\infty$-category and $\Mnd \to \KK$ is a homotopy coherent monad on $B$, then the $\infty$-category of $T$-algebras in $B$ may be recovered as an appropriately defined \emph{flexible weighted limit} of the diagram $\Mnd \to \KK$. This limit computes the value of the  right Kan extension along $\Mnd\hookrightarrow\Adj$ at the object $-$ and so in fact constructs the entire \emph{monadic}  homotopy coherent adjunction. All of the examples of quasi-categorically enriched categories $\KK$ mentioned above are $\infty$-\emph{cosmoi}, in which such limits exist. 

A full description of the $\infty$-category of algebras for a homotopy coherent monad is given in \cite[\S 6]{RV-II}, but we can at least give an informal description here. To a rough approximation, a \emph{homotopy coherent $T$-algebra} for a homotopy coherent monad acting on an object $B \in \KK$ is a homotopy coherent diagram of shape $\DDelta_\top$ in $B$ satisfying various functoriality conditions, which are suggested by the picture
\[
\begin{tikzcd}
b \arrow[r, shift left=0.25em, "\eta"] & Tb \arrow[l, shift left=0.25em, "\beta"] \arrow[r, shift right=0.5em, "T\eta" description] \arrow[r, shift left=1.25em, "\eta T"] & T^2b \arrow[l, shift right=0.5em, "\mu" description] \arrow[l, shift left=1.25em, "T\beta"] \arrow[r, shift left=0.25em] \arrow[r, shift right=.75em] \arrow[r, shift left = 1.25em] & T^3b \arrow[l, shift left=0.25em] \arrow[l, shift right=0.75em] \arrow[l, shift left=1.25em] & \cdots 
\end{tikzcd}
\]

\subsection{Other vistas}

Homotopy coherent adjunctions, monads, and algebras represent a rather small part of ``higher algebra,'' which includes strong shape theory \cite{PS}, Waldhausen's notion of ``brave new rings'' \cite{vogt2}, the $\infty$-operads of Cisinski-Moerdijk-Weiss \cite{CM,MW} and Lurie \cite{lurie}, among other topics. A wonderful survey can be found in \cite{Gepner}. 

\refs

\bibitem[Boardman--Vogt, 1973]{BV} J.~M.~Boardman, R.~M.~Vogt, Homotopy invariant algebraic structures on topological spaces, {\em Lecture Notes in Mathematics} {\bf 347}, 1973.

\bibitem[Cisinski--Moerdijk, 2011]{CM} D.-C.~Cisinski, I.~Moerdijk, Dendroidal sets as models for homotopy operads, {\em J.~Topology}, {\bf 4(2)} (2011) 257--299.

\bibitem[Cordier--Porter, 1986]{CP-vogt} J.~M.~Cordier, T.~Porter, Vogt's theorem on categories of homotopy coherent diagrams, Math.~Proc.~Camb.~Phil.~Soc. {\bf 100} (1986) 65--90.

\bibitem[Cordier--Porter, 1988]{CP-maps} J.~M.~Cordier, T.~Porter, Maps between homotopy coherent diagrams, Topology and its Applications {\bf 28} (1988) 255--275.

\bibitem[Dugger--Spivak, 2011]{dugger-spivak} D. Dugger and D.~I.~Spivak, Rigidification of quasi-categories, {\em Algebr. Geom. Topol.,} 11(1):225--261, 2011.

\bibitem[Dwyer--Kan, 1980]{DK-simplicial} W.~G.~Dwyer, D.~M.~Kan, Simplicial localizations of categories, J. Pure Appl.~Algebra {\bf 17} (1980), 267--284.

\bibitem[Dwyer--Kan--Smith, 1989]{DKS-homotopy} W.~G.~Dwyer, D.~M.~Kan, J.~H.~Smith, Homotopy commutative diagrams and their realizations, Journal of Pure and Applied Algebra {\bf 57} (1989) 5--24.

\bibitem[Gabriel--Zismann, 1967]{GZ} P. Gabriel and M. Zisman, Calculus of fractions and homotopy theory, Ergebnisse der Mathematik und ihrer Grenzgebiete, Band 35. Springer-Verlag New York, Inc., New York, 1967.

\bibitem[Gepner, 2019]{Gepner} D.~Gepner, An introduction to higher categorical algebra, in {\em Handbook of Homotopy Theory}, Haynes Miller ed., Chapman and Hall/CRC, 2019, 487--548.

\bibitem[Joyal--Tierney, 2007]{JT} A. Joyal and M.~Tierney, Quasi-categories vs Segal spaces. {\em Categories in algebra, geometry and
mathematical physics} Contemp. Math.~{\bf 431} (2007), 277--326. 

\bibitem[Lurie, 2017]{lurie} J.~Lurie, \emph{Higher Algebra}, 2017, \href{https://www.math.ias.edu/~lurie/papers/HA.pdf}{www.math.ias.edu/$\sim$lurie/papers/HA.pdf}

\bibitem[Meyer, 1984]{meyer} J.-P.~Meyer. Bar and cobar constructions. I. {\em J. Pure Appl. Algebra}, 33(2):163--207, 1984.

\bibitem[Moerdijk--Weiss, 2007]{MW} I.~Moerdijk, I.~Weiss, Dendroidal sets, {\em Algebraic \& Geometric Topology}, {\bf 7}, (2007), 1441--1470.

\bibitem[Porter--Stasheff, 2022]{PS} T.~Porter, J.~Stasheff,  Homotopy coherent representations, 2022, \href{https://arxiv.org/abs/2202.05322}{arXiv:2202.05322}

\bibitem[Quillen, 1967]{quillen} D.~G.~Quillen, \emph{Homotopical algebra}, Lecture Notes in Math.~{\bf 43}, Springer, 1967.

\bibitem[Riehl, 2011]{riehl-leisurely} E.~Riehl, A leisurely introduction to simplicial sets, preprint available from \href{https://emilyriehl.github.io/files/ssets.pdf}{emilyriehl.github.io/files/ssets.pdf}

\bibitem[Riehl, 2011]{riehl-necklace} E.~Riehl, On the structure of simplicial categories associated to quasi-categories, {\em Math.~Proc.~Camb.~Phil.~Soc.} {\bf 150} (2011), no.3., 489--504. \href{http://arxiv.org/abs/0912.4809}{arXiv:0912.4809}

\bibitem[Riehl, 2014]{riehl-cathtpy} E.~Riehl, \emph{Categorical homotopy theory}, New Mathematical Monographs~24, Cambridge University Press, 2014.

\bibitem[Riehl--Verity, 2016]{RV-II} E.~Riehl and D.~Verity, Homotopy coherent adjunctions and the formal theory of monads, {\em Adv.~Math}~{\bf 286} (2016), 802--888.   \href{http://arxiv.org/abs/1310.8279}{arXiv:1310.8279}

\bibitem[Riehl--Verity, 2018]{RV-VI} E.~Riehl and D.~Verity, The comprehension construction,{\em Higher Structures} {\bf 2} (2018), no.~1, 116-190,  \href{http://arxiv.org/abs/1706.10023}{arXiv:1706.10023}

\bibitem[Riehl--Verity, 2020]{riehl-scratch}  E.~Riehl and D.~Verity, {$\infty$-category theory from scratch}, {\em Higher Structures} 4(1):115--167, 2020,  \href{http://arxiv.org/abs/1608.05314}{arXiv:1608.05314}

\bibitem[Riehl--Verity, 2021]{RV-IX} E.~Riehl and D.~Verity,  Cartesian exponentiation and monadicity, 2021, \href{https://arxiv.org/abs/2101.09853}{arXiv:2101.09853}

\bibitem[Riehl--Verity, 2022]{RV-elements} E.~Riehl and D.~Verity, {\em Elements of $\infty$-Category Theory}, Cambridge Studies in Advanced Mathematics 194, Cambridge University Press, 2022.

\bibitem[Schanuel--Street, 1986]{SS}S.~Schanuel and R.H.~Street, The free adjunction {\em Cah. Topol. G\'{e}om. Diff\'{e}r. Cat\'{e}g.}~{\bf 27} (1986), 81--83.

\bibitem[Vogt, 1973]{vogt1} R.M.~Vogt, Homotopy limits and colimits, Math.~Z. {\bf 143} 11--52 (1973).

\bibitem[Vogt, 1999]{vogt2} R.M.~Vogt, Introduction to algebra over ``brave new rings''. Proceedings of the 18th Winter School ``Geometry and Physics.'' Palermo: Circolo Matematico di Palermo, 1999. 49--82. 

\endrefs

\end{document}